\documentclass[a4paper]{article}

\usepackage[english]{babel}
\usepackage[utf8x]{inputenc}
\usepackage[T1]{fontenc}

\usepackage[a4paper,top=2cm,bottom=2.5cm,left=2.3cm,right=2.3cm,marginparwidth=1.75cm]{geometry}

\usepackage[colorinlistoftodos]{todonotes}
\usepackage[colorlinks=true, allcolors=blue]{hyperref}
\usepackage{xy}
\usepackage[affil-it]{authblk}
\usepackage{amssymb}
\usepackage{cancel}
\usepackage{epstopdf}
\usepackage{verbatim}
\usepackage{amsmath,amscd}
\usepackage{graphicx}
\usepackage{subcaption}
\usepackage{amsthm}
\usepackage{tikz}
\usetikzlibrary{cd}
\usepackage{pgf}
\usepackage{mathrsfs}
\usetikzlibrary{arrows}
\usepackage[toc,page]{appendix}
\usepackage{caption}
\usepackage[colorinlistoftodos]{todonotes}

\makeatletter
\newtheorem*{rep@theorem}{\rep@title}
\newcommand{\newreptheorem}[2]{%
\newenvironment{rep#1}[1]{%
 \def\rep@title{#2 \ref{##1}}%
 \begin{rep@theorem}}%
 {\end{rep@theorem}}}
\makeatother

\newtheorem{theorem}{Theorem}[section]
\newtheorem{lemma}{Lemma}[section]
\newreptheorem{theorem}{Theorem}

\theoremstyle{definition}
\newtheorem{mydef}{Definition}[section]

\newtheorem{rem}{Remark}[section]

\title{Semisimple pointed isogeny graphs for abelian varieties}
\author{Paul Alexander Helminck}
\affil{University of Bremen,\\
ALTA, Institute for Algebra, Geometry, Topology and their Applications}

\begin{document}
\maketitle

\definecolor{qqqqff}{rgb}{0,0,1}

\begin{abstract}
In this paper we show that if $\phi_{i}:A_{i}\rightarrow{A}$ is a semisimple pointed $K$-rational $\ell$-isogeny graph of order $n$ for a prime $\ell$, then the group of $\ell$-torsion points $A[\ell](\overline{K})$ contains a subspace of dimension $n$ generated by $K$-rational points. 
We also show that the same result is true for elliptic curves without the semisimplicity condition. Furthermore, we give an explicit counterexample for abelian varieties of higher dimension to show that the semisimplicity condition is indeed necessary.  
\end{abstract}

\section{Introduction}

Let $\ell$ be a prime. Let $A$ be an abelian variety over a field $K$ of characteristic $p\neq{\ell}$ and let $G=\mathrm{Gal}(K(A[\ell])/K)$. 
Suppose we have $n$ isogenies $\phi_{i}:A_{i}\rightarrow{A}$ over $K$, whose kernels are generated by $K$-rational points of order $\ell$. We refer to such a collection of isogenies as a $K$-rational pointed $\ell$-isogeny graph. If the corresponding injections of function fields $K(A)\rightarrow{K(A_{i})}$ are linearly disjoint, we say that the pointed isogeny graph is of order $n$. 

Elliptic curves give rise to the simplest examples of nontrivial pointed isogeny graphs. Here we have the following theorem: 

\begin{reptheorem}{MainTheorem1}
Let
\begin{center}
\begin{tikzcd}[every arrow/.append style={dash}]
E_{1} \ar[rightarrow]{r}{\phi_{1}} & E \ar[leftarrow]{r}{\phi_{2}} & E_{2} 
\end{tikzcd}
\end{center}
be a $K$-rational pointed $\ell$-isogeny graph of order two. 
Then $E[\ell](\overline{K})\simeq{(\mathbb{Z}/\ell\mathbb{Z})^{2}}$ as $G_{K}$-modules. 
\end{reptheorem}
For abelian varieties of dimension $g>1$, the situation is slightly more complicated. Here we will assume that the representation 
\begin{equation}
\rho_{\ell}:G_{K}\rightarrow{\mathrm{GL}_{2g}(\mathbb{F}_{\ell})}
\end{equation}
is \emph{semisimple}, in the sense that $G_{K}$-invariant subspaces admit $G_{K}$-invariant complements. This condition is satisfied for instance when $\ell\nmid{|G|}$, where $G=G_{K}/\mathrm{ker}(\rho_{\ell})$.  
In the semisimple case, we have the following theorem:
\begin{reptheorem}{MainTheorem2}
Let $\phi_{i}:A_{i}\rightarrow{A}$ be a \emph{semisimple} pointed $K$-rational $\ell$-isogeny graph of order $n$. 
There then exists an injection of $G_{K}$-modules
\begin{equation}
(\mathbb{Z}/\ell\mathbb{Z})^{n}\rightarrow{A[\ell]}(\overline{K}).
\end{equation}
\end{reptheorem}
Here $(\mathbb{Z}/\ell\mathbb{Z})^{n}$ has the trivial $G_{K}$-action.
The semisimplicity condition is indeed necessary as we will see in Section \ref{AbelianVarieties}, where we give an explicit counterexample over $\mathbb{Q}$ using products of elliptic curves. 

The problem arose from considering descent problems and their applications to finding number fields $K$ with a class group having a high $\ell$-part (see \cite{Paul3}).  
Given a $K$-rational point on $A$, one can pull this point back through all $n$ isogenies. This gives $n$ extensions and by class field theory one obtains a part of a generalized class group. Imposing some extra conditions can make most of these extensions unramified everywhere, which then gives subgroups of the ordinary class group. 

The paper is structured as follows: we begin by introducing the necessary concepts in Section \ref{Preliminaries}. After this, we will consider the case of elliptic curves in Section \ref{EllipticCurves} and prove Theorem \ref{MainTheorem1}. In Section \ref{AbelianVarieties}, we consider abelian varieties of dimension greater than one and prove Theorem \ref{MainTheorem2}. We will then give a counterexample to show that the semisimplicity condition in Theorem \ref{MainTheorem2} is necessary. 

\section{Preliminaries}\label{Preliminaries}

Let $\ell$ be a prime and let $K$ be a field of characteristic $p\neq{\ell}$. Let $K^{s}$ be a separable closure of $K$ with Galois group $G_{K}:=\mathrm{Gal}(K^{s}/K)$. We refer the reader to \cite{Milne1}, \cite{GeerMoon} and \cite{Mumf1} for an introduction to abelian varieties. We will give a quick summary of the facts needed here. 

An abelian variety over $K$ is an algebraic group $A$ that is geometrically integral and proper over $K$. Such a variety $A$ is always commutative and projective, see \cite[Corollary 1.4 and Theorem 6.4]{Milne1}. For any $n\in\mathbb{Z}_{>{0}}$, we let $A[n]$ be the kernel of the multiplication map $[n]:A\rightarrow{A}$. We then have
\begin{lemma}\label{SeparableLemma}
Let $A$ be an abelian variety of dimension $g$ over $K$. If $p\nmid{n}$, then $A[n]$ is \'{e}tale over $K$ and $A[n](K^{s})\simeq{(\mathbb{Z}/n\mathbb{Z})^{2g}}$.
\end{lemma}
\begin{proof}
See \cite[Chapter 7, Theorem 4.38]{liu2}. Note that the fact that $A[n]$ is \'{e}tale over $K$ implies that the geometric points $A[n](\overline{K})$ are defined over $K^{s}$.
\end{proof}
We now specialize to a prime $n=\ell$ and we will use the notation $A[\ell](\overline{K})$, which is equal to $A[\ell](K^{s})$ by Lemma \ref{SeparableLemma}. The group $A[\ell](\overline{K})$ is naturally an $\mathbb{F}_{\ell}$-vector space of dimension $2g$ and for any subspace $V\subset{A[\ell](\overline{K})}$, we denote its dimension by $\mathrm{dim}(V)$.



An isogeny of abelian varieties is a homomorphism $\phi:A\rightarrow{A'}$ such that $\phi$ is finite and surjective. We define the \emph{degree} of $\phi$ by $\mathrm{deg}(\phi):=|(\ker{\phi})(\overline{K})|$. There is then a \emph{dual} isogeny $\hat{\phi}:A'\rightarrow{A}$ such that $\hat{\phi}\circ{\phi}=[m]_{A}$. We will only be interested in separable isogenies in this paper. That is, isogenies for which the injection of function fields $K(A')\rightarrow{K(A)}$ is a separable field extension. 
%
For any $K$-rational point $P\in{A[\ell](K)}$, we consider the corresponding isogeny $\phi:A\rightarrow{A/<P>}$ over $K$. 
Let us write $A'=A/<P>$. Then $\phi$ induces a finite injection of function fields $K(A')\rightarrow{K(A)}$ of degree $\ell$.
\begin{mydef}\label{PointedIsogenyGraph}
Let $A_{i}/K$ be a finite collection of abelian varieties over $K$ with $K$-rational points $P_{i}\in{A_{i}[\ell](K)}$ and assume that they induce isogenies to a single abelian variety $A$. That is, we have isogenies $\phi_{i}:A_{i}\rightarrow{A}$. We then say that they form a pointed $K$-rational $\ell$-isogeny graph, or just: pointed isogeny graph. 
\end{mydef}
\begin{mydef}\label{OrderIsogenyGraph}
Let $\phi_{i}:A_{i}\rightarrow{A}$ be a set of $n$ isogenies forming a pointed $K$-rational $\ell$-isogeny graph. We say that the pointed isogeny graph is of order $n$ if the field extensions
\begin{equation}
K(A)\rightarrow{K(A_{i})}
\end{equation} 
are linearly disjoint.
\end{mydef}
\begin{rem}
We have the following equivalent conditions for a pointed isogeny graph to be of order $n$: 
\begin{enumerate}
\item The field $K_{n}:=\prod_{i=1}^{n}K(A_{i})$ has dimension $\ell^{n}$ as a vector space over $K(A)$.
\item The induced map $\bigotimes_{i=1}^{n}K(A_{i})\rightarrow{K_{n}}$ is an injection.
\end{enumerate}
\end{rem}

We have a natural action of $G_{K}$ on $A[\ell](\overline{K})$. This gives a representation
\begin{equation}
\rho_{\ell}:G_{K}\rightarrow{\mathrm{GL}_{2g}(\mathbb{F}_{\ell})}.
\end{equation}
We say that $\rho_{\ell}$ is \emph{semisimple} if every $G_{K}$-invariant subspace $V\subset{\mathbb{F}_{\ell}^{2g}}=A[\ell](\overline{K})$ has a $G_{K}$-invariant complement. That is, there exists a $G_{K}$-invariant subspace $V'$ such that $V\oplus{V'}=\mathbb{F}_{\ell}^{2g}$. Here $G_{K}$ acts through the representation $\rho_{\ell}$.  
\begin{mydef}
A pointed isogeny graph $\phi_{i}:A_{i}\rightarrow{A}$ is said to be \emph{semisimple} if the representation $\rho_{\ell}$ is semisimple. 
\end{mydef}
\begin{rem}
The representation is semisimple for example when $\ell\nmid{|G|}$, where $G=\mathrm{Gal}(K(A[\ell])/K)$, see \cite[Chapter 1, Theorem 1]{Serre1977}.
\end{rem}

\newpage
\section{Pointed isogeny graphs for elliptic curves}\label{EllipticCurves}

Let $E/K$ be an elliptic curve over a field $K$ of characteristic $p\neq{\ell}$. We will prove the following 
\begin{theorem}\label{MainTheorem1}
Let
\begin{center}
\begin{tikzcd}[every arrow/.append style={dash}]
E_{1} \ar[rightarrow]{r}{\phi_{1}} & E \ar[leftarrow]{r}{\phi_{2}} & E_{2} 
\end{tikzcd}
\end{center}
be a $K$-rational pointed $\ell$-isogeny graph of order two. 
Then $E[\ell](\overline{K})\simeq{(\mathbb{Z}/\ell\mathbb{Z})^{2}}$ as $G_{K}$-modules. 
\end{theorem}
\begin{proof}
Let $\{P_{1},Q_{1}\}$ be a basis for $E_{1}[\ell](\overline{K})$ and $\{P_{2},Q_{2}\}$ a basis for $E_{2}[\ell](\overline{K})$ such that $\mathrm{ker}(\phi_{i})=<P_{i}>$. We claim that $R_{1}:=\phi_{1}(Q_{1})$ and $R_{2}:=\phi_{2}(Q_{2})$ are linearly independent and $K$-rational. Write $H_{i}=<R_{i}>$. 
\begin{lemma}\label{IndependentLemma1}
\begin{equation}
H_{1}\neq{H_{2}}.
\end{equation}
\end{lemma}
\begin{proof}
Let $H=H_{1}\cap{H_{2}}$. Since $H\subset{H_{i}}$, we have induced morphisms
\begin{equation}
E\rightarrow{E/H}\rightarrow{E/H_{i}}\rightarrow{E},
\end{equation}
where the last arrow is given by $\phi_{i}$. This gives rise to injections of function fields
\begin{equation}
K(E)\rightarrow{K(E_{i})}\rightarrow{K(E/H)}\rightarrow{K(E)}.
\end{equation}
By assumption, we have that $K(E_{1})\cdot{K(E_{2})}$ has dimension $\ell^{2}$ over $K(E)$. In other words, the last injection is an isomorphism and thus $H=(1)$, giving the Lemma.
\end{proof}

\begin{lemma}
The $R_{i}$ are $K$-rational.
\end{lemma}
\begin{proof}
Let $\sigma\in{G_{K}}$. Consider $R'_{1}:=\hat{\phi}_{2}(R_{1})$. Then $R'_{1}\in\mathrm{ker}(\phi_{2})$, meaning that it is $K$-rational. This gives $\sigma(R'_{1})=R'_{1}$ and thus $\sigma(R_{1})-R_{1}\in{H_{2}}$. But $H_{1}=<R_{1}>$ is $G_{K}$-invariant (as it is the kernel of $\hat{\phi}_{1}$) and thus $\sigma(R_{1})-R_{1}\in{H_{1}\cap{H_{2}}}=(1)$. Similarly, we have $\sigma(R_{2})=R_{2}$, yielding the Lemma.
\end{proof}
In other words, we now have two $K$-rational points $R_{i}\in{E[\ell](K)}$ which are linearly independent by Lemma \ref{IndependentLemma1}. This gives the theorem. 

\end{proof}

\section{Semisimple pointed isogeny graphs}\label{AbelianVarieties}

We now turn to pointed $K$-rational $\ell$-isogeny graphs for abelian varieties of dimension $g>1$. These isogeny graphs $A_{i}\rightarrow{A}$ can exist without an injection of $G_{K}$-modules  $(\mathbb{Z}/\ell\mathbb{Z})^{n}\rightarrow{A[\ell](\overline{K})}$, as we will see after Theorem \ref{MainTheorem2}. If the representation 
\begin{equation}
\rho_{\ell}:G_{K}\rightarrow{\mathrm{GL}_{2g}(\mathbb{F}_{\ell})}
\end{equation}
is \emph{semisimple} however, we have the following theorem: 

\begin{theorem}\label{MainTheorem2}
Let $A$ be an abelian variety of dimension $g>1$ and let $A_{i}\rightarrow{A}$ be a semisimple $K$-rational $\ell$-isogeny graph of order $n$. There then exists an injection
\begin{equation}
(\mathbb{Z}/\ell\mathbb{Z})^{n}\rightarrow{A[\ell](\overline{K})}
\end{equation}
of $G_{K}$-modules.
\end{theorem}
\begin{proof}
Let $\{P_{1,i},...,P_{2g,i}\}$ be a basis of $A_{i}[\ell](\overline{K})$ such that $\mathrm{ker}(\phi_{i})=<P_{1,i}>$.
Define
\begin{equation}
H_{i}:=\phi_{i}(\text{Span}\{P_{2,i},P_{3,i},...,P_{2g,i}\})=\phi_{i}(A_{i}[\ell](\overline{K})). 
\end{equation}
Note that $H_{i}$ is $G_{K}$-invariant, since it is the kernel of the dual isogeny $\hat{\phi}_{i}$. 
Let $J$ be any nonempty subset of $I:=\{1,2,...,n\}$. We define 
\begin{equation}
H_{J}:=\bigcap_{j\in{J}}H_{j}. 
\end{equation}
Thus if $J_{1}\subseteq{J_{2}}$ then 
\begin{equation}
H_{J_{1}}\supseteq{H_{J_{2}}}. 
\end{equation}
Note that any $H_{J}$ is again $G_{K}$-invariant, since it is an intersection of $G_{K}$-invariant subspaces. 

We will start by proving the following
\begin{lemma}\label{DimensionLemma2}
Let $J\subseteq{I}=\{1,..,n\}$ and consider the subgroup $H_{J}$. Then
\begin{equation}
\mathrm{dim}(H_{J})=2g-\#{J}.
\end{equation}
\end{lemma}
\begin{proof}
We will prove this by induction on the order of $J$. For $J=\{i\}$, we have $\mathrm{dim}(H_{i})=2g-1$ by definition of the $H_{i}$. Suppose that the statement is true for subsets of order $k$ and let $J$ be a subset of order $k+1$. Consider any $j\in{J}$ and let $J'=J\backslash\{j\}$. We consider two cases:
\begin{enumerate}
\item $H_{J'}\subseteq{H_{j}}$,
\item $H_{J'}\subsetneq{H_{j}}$.
\end{enumerate}
Suppose that $H_{J'}\subseteq{H_{j}}$. We then have the following induced injections of function fields
\begin{equation}
K(A)\rightarrow{K(A_{j})}\rightarrow{K(A/H_{J'})}\rightarrow{K(A)}.
\end{equation} 
Note that $K(A/H_{J'})/K(A)$ has degree $\ell^{|J'|}$ by the induction hypothesis. Note also that the product 
\begin{equation}
\prod_{i\in{J}}K(A_{j})
\end{equation}
is contained in $K(A/H_{J'})$ and has degree $\ell^{|J|}$. But $\ell^{|J|}>\ell^{|J'|}$, yielding a contradiction. We thus conclude that $H_{J'}\subsetneq{H_{j}}$. There is then an $x\in{H_{J'}}$ such that $x\notin{H_{j}}$. But then $<x>+H_{j}=A[\ell](\overline{K})$, since $H_{j}$ is $2g-1$-dimensional. We then also have that $H_{J}+H_{j}=A[\ell](\overline{K})$, since $<x>\subseteq{H_{J}}$. Applying the inclusion/exclusion formula, 
 we obtain
\begin{equation}
\mathrm{dim}(H_{J})=2g-\mathrm{dim}(H_{j})+\mathrm{dim}(H_{J'})=2g-(2g-1)+\mathrm{dim}(H_{J'})=1+\mathrm{dim}(H_{J'}).
\end{equation}
This then gives $\mathrm{dim}(H_{J})=1+\mathrm{dim}(H_{J'})=1+2g-|J'|=2g-|J|$, as desired. By induction, we conclude that the formula holds for all subsets of $I$, yielding the Lemma.
\end{proof}

Now consider the inclusions $H_{I\backslash{i}}\supset{H_{I}}$, where $I=\{1,...,n\}$. By applying the semisimplicity assumption to 
every inclusion, we obtain $G_{K}$-invariant subspaces $W_{i}$ such that 
\begin{equation}\label{ConstructionEquation}
H_{I}\oplus{W_{i}}=H_{I\backslash{i}}.
\end{equation}
By Lemma \ref{DimensionLemma2}, we then find that $\mathrm{dim}(W_{i})=1$ and we can thus write $W_{i}=<Q_{i}>$ for $Q_{i}\in{A[\ell](\overline{K})}$. Note that $W_{i}\cap{H_{i}}=(1)$ for every $i$. 
\begin{lemma}
The $Q_{i}$ are $K$-rational.
\end{lemma}
\begin{proof}
Consider the points $Q'_{i}:=\hat{\phi}_{i}(Q_{i})\in{A_{i}[\ell](\overline{K})}$. They are $K$-rational, since they are in the kernel of $\phi_{i}$. We thus find that $\sigma(Q_{i})-Q_{i}\in\mathrm{ker}(\hat{\phi}_{i})=H_{i}$. But the $W_{i}$ are invariant, so $\sigma(Q_{i})-Q_{i}\in{W_{i}}\cap{H_{i}}=(1)$. This yields $\sigma(Q_{i})=Q_{i}$, as desired. 
\end{proof}

\begin{lemma}
The $Q_{i}$ are linearly independent.
\end{lemma}
\begin{proof}
We will prove this by induction on $i$. For $i=1$, we have that $Q_{1}$ is not trivial by construction (see Equation \ref{ConstructionEquation}). We now suppose that the Lemma is true for $i$ and we will show that $Q_{i+1}$ is independent of the other $Q_{k}$ for $k\leq{i}$. Suppose for a contradiction that
\begin{equation}
Q_{i+1}=\sum_{j=1}^{i}c_{j}Q_{j}.
\end{equation}
By construction, we find that $Q_{j}\in{H_{i+1}}$ for every $j\leq{i}$. But then $Q_{i+1}\in{H_{i+1}}$, a contradiction. By induction, we conclude that the $Q_{i}$ are linearly independent.

\end{proof}
We now have a set $\{Q_{i}\}$ of $n$ linearly independent, $K$-rational points in $A[\ell](\overline{K})$. These give a natural injection of $G_{K}$-modules $(\mathbb{Z}/\ell\mathbb{Z})^{n}\rightarrow{A[\ell](\overline{K})}$, which finishes the proof.
\end{proof}

We note that the condition on the Galois group in Theorem \ref{MainTheorem2} is necessary by the following counterexample. Let $\phi_{1}:E'_{1}\rightarrow{E_{1}}$ and $\phi_{2}:E'_{2}\rightarrow{E_{2}}$ be two isogenies of degree $\ell$ whose kernels are generated by $K$-rational points. Suppose furthermore that $E_{1}$ and $E_{2}$ contain no $K$-rational $\ell$-torsion points. 
We then have the isogeny graph
\begin{center}
\begin{tikzcd}[every arrow/.append style={dash}]
E'_{1}\times{E_{2}}\ar[rightarrow]{r}{(\phi_{1},id)} & E_{1}\times{E_{2}}\ar[leftarrow]{r}{(id,\phi_{2})} & E_{1}\times{E'_{2}}. 
\end{tikzcd}
\end{center}

Note that by construction, $(E_{1}\times{E_{2}})[\ell]$ does not contain any $K$-rational points. More explicitly, let $\ell=3$ and $K:=\mathbb{Q}$. We then have the following universal families
\begin{align*}
E_{3}&:y^2+wxy+vy=x^3,\\
E'_{3}&:y^2+wxy+vy=x^3-5wvx-v(w^3+7v).
\end{align*}
Here $P:=(0,0)\in{E_{3}(\mathbb{Q}(w,v))}$ is a point of order three and $E'_{3}=E_{3}/<P>$. Taking values $(v_{0},w_{0})\in\mathbb{Q}^{2}$ of $(v,w)$ such that $E'_{3}$ has no rational three torsion then yields the desired curve. For instance, $v_{0}=2$ and $w_{0}=1$ work. Taking $E'_{1}=E'_{2}$ and $E_{1}=E_{2}=E'_{1}/<P>$ then yields the explicit counterexample over $\mathbb{Q}$. The formulas for the universal families can be found in \cite[Page 15]{Paul3}.

\bibliographystyle{alpha}
\bibliography{bibfiles}{}

\end{document}